\documentclass[lettersize,onecolumn]{IEEEtran}
\usepackage{amssymb,amsmath,latexsym}
\usepackage{amsthm}
\usepackage[shortlabels]{enumitem}
\usepackage{caption}
\usepackage{xcolor}
\usepackage{float}
\usepackage{graphicx}
\usepackage{tikz}
\usepackage{url}
\usepackage{breakurl}
\usepackage[colorlinks=true,citecolor=blue,linkcolor=blue,urlcolor=blue,bookmarks,bookmarksopen,bookmarksdepth=2,backref=page,breaklinks]{hyperref}

\usepackage{bookmark}
\bookmarksetup{
	numbered, 
	open,
}

\renewcommand*{\backref}[1]{}
\renewcommand*{\backrefalt}[4]{%
	\ifcase #1 (Not cited.)%
	\or        (Cited on page~#2.)%
	\else      (Cited on pages~#2.)%
	\fi}

\newcommand{\F}{\mathbb F}

\DeclareMathOperator{\AGL}{AGL}

\DeclareMathOperator{\GL}{GL}

\DeclareMathOperator{\PG}{PG}

\theoremstyle{plain}
\newtheorem{theorem}{Theorem}[section]
\newtheorem{lemma}[theorem]{Lemma}
\newtheorem{proposition}[theorem]{Proposition}

\newtheorem{corollary}[theorem]{Corollary}

\theoremstyle{definition}

\newtheorem{example}[theorem]{Example}
\newtheorem{openproblem}[theorem]{Open Problem}
\newtheorem{definition}[theorem]{Definition}
\newtheorem{remark}[theorem]{Remark}

\newtheorem{construction}[theorem]{Construction}

\numberwithin{theorem}{section}
\numberwithin{equation}{section}
\numberwithin{table}{section}
\numberwithin{figure}{section}

\DeclareMathOperator{\image}{Im}
\begin{document}
\title{On the Walsh spectra of quadratic APN functions}

\author{Sophie Hannah Bénéteau, Nicolas Goluboff, Lukas K\"olsch, and Divyesh Vaghasiya\thanks{
S.B. is with University of Florida (email: \url{sbenetea@icloud.com}), N.G. is with University of Massachusetts Amherst (email: \url{nico.goluboff@gmail.com}), L.K and D.V. are with University of South Florida  (email: \url{lukas.koelsch.math@gmail.com}, \url{divyeshvaghasiya00@gmail.com})}}
\maketitle
\begin{abstract} 
APN functions provide optimal resistance to differential attacks when used as building blocks in  block ciphers and are thus of great theoretical interest. One of the most important properties of APN functions is their linearity, which is directly related to the Walsh spectrum of the function. In this paper, we establish two novel connections that allow us to derive strong conditions on the Walsh spectra of quadratic APN functions. We prove that the Walsh transform of a quadratic APN function $F$ operating on $n=2k$ bits is uniquely associated with a vector space partition of $\F_2^n$ and a specific blocking set in the corresponding projective space $\PG(n-1,2)$. These connections allow us to prove a variety of results on the Walsh spectrum of $F$. We prove for instance that $F$ can have at most one component function of amplitude larger than $2^{3n/4}$.  We also find the first nontrivial upper bound on the number of bent component functions of a quadratic APN function, and provide conditions for a function to be CCZ-equivalent to a permutation based on its number of bent components.
\end{abstract}



\section{Introduction and Notations}

Vectorial Boolean functions play an important role in the design of symmetric cryptosystems, serving as the building blocks of block ciphers. To resist differential attacks, a vectorial Boolean function should have low differential uniformity~\cite{biham1991differential,nyberg1992provable}.
\begin{definition}[Differential uniformity]
	A function $F \colon \F_{2}^n \rightarrow \F_{2}^n$ has differential uniformity $d$, if 
	\begin{equation*}
		d=\max_{a\in \F_{2}^n\setminus\{0\}, b\in \F_{2}^n} |\{x \colon F(x)+F(x+a)=b\}|.
	\end{equation*}
	A function with differential uniformity $2$ is called almost perfect nonlinear (APN).
\end{definition}
As the differential uniformity is always even, APN functions $F \colon \F_{2}^n \rightarrow \F_{2}^n$ provide optimal resistance against differential attacks. One of the most important properties of vectorial Boolean functions for use in substitution permutation networks (SPNs) is their linearity, which can be defined via the Walsh transform of the function. For applications, the case where $n$ is even is of particular importance, and we will mostly deal with this case in this paper.

We recall some notation and basic definitions; we refer the reader to the standard work~\cite{Carlet2021_Book}. We use the notation $\langle \cdot , \cdot \rangle_k$ for the usual dot product on $\F_2^k$, and omit the subscript $k$ when clear from the context. If $F \colon \F_2^n \rightarrow \F_2^m$, then we call the $2^m$ Boolean functions $F_b \colon \F_2^n \rightarrow \F_2$ defined via $F_b(x)=\langle b,F(x)\rangle_m$ for $b \in \F_2^m$ the \emph{component functions} of $F$, and we call $F_0$ the trivial component function. 
\begin{definition}[Walsh transform and linearity]
	Let $F \colon \F_{2}^n \rightarrow \F_{2}^m$ be a function. We define its \emph{Walsh transform}
	\begin{equation*}
		W_F(b,a)=\sum_{x \in \F_2^n}(-1)^{F_b(x)+\langle x,a \rangle_n}.
	\end{equation*}
    We call the multiset that collects the values $W_F(b,a)$ for all $a\in \F_2^n$, $b \in \F_2^m$ the \emph{Walsh spectrum} and the multiset that collects the values $|W_F(b,a)|$ for all $a\in \F_2^n$, $b \in \F_2^m$ the \emph{extended Walsh spectrum} of $F$.
	The \emph{linearity} of $F$ is then defined as $$L(F)=\max_{a \in \F_2^n,b\in \F_2^m\setminus\{0\}}|W_F(b,a)|.$$
\end{definition}

A function $F \colon \F_2^n \rightarrow \F_2^m$ is called \emph{quadratic} if all component functions of $F$ are a sum of a quadratic form and an affine function over $\F_2$. It follows from the theory of quadratic forms that for a quadratic function $F \colon \F_2^n \rightarrow \F_2^m$ and a fixed $b$, we have $|W_F(b,a)| \in \{0,2^{\frac{n+k}{2}}\}$ for some $0 \leq k \leq n$. Parseval's equation states that $\sum_{a \in \F_2^n}|W_F(b,a)|^2=2^{2n}$, which then shows (fixing $b$ and ranging over $a \in \F_2^n$) that $|W_F(b,a)|$ takes on the value $2^{\frac{n+k}{2}}$ exactly $2^{n-k}$ times and the value $0$ exactly $2^n-2^{n-k}$ times. We call $2^{\frac{n+k}{2}}$ the \emph{amplitude} of the component function $F_b$. Components with the lowest possible amplitude $2^{n/2}$ are of special importance (they are \emph{bent} functions), and we call these components consequently the \emph{bent components} of $F$. The extended Walsh spectrum, as well as the linearity of a quadratic function $F$, are thus uniquely determined by the amplitudes of the component functions of $F$. 

We thus focus our attention in this article on the \emph{amplitude distribution} of quadratic APN functions. It is easy to see that the trivial component $F_0$ always has amplitude $2^n$, and we will exclude it from our considerations.

\begin{definition}[Amplitude distribution] \label{def:amdis}
    Let $F \colon \F_2^n \rightarrow \F_2^m$ be a quadratic function. We call the multiset of all amplitudes of non-trivial components of $F$ the \emph{amplitude distribution} of $F$. If $F$ has $k_i$ non-trivial component functions of amplitude $2^{\frac{n+i}{2}}$ then we write its amplitude distribution as $[0^{k_0},1^{k_1},2^{k_2},\dots,(n-1)^{k_{n-1}},n^{k_n}]$.
\end{definition}

\begin{example}
    Let $n$ be even. A quadratic APN function $F \colon \F_2^n \rightarrow \F_2^n$ is said to have \emph{classical Walsh spectrum}, if it has $\frac{2}{3}(2^n-1)$ bent components and $\frac{1}{3}(2^n-1)$ components with amplitude $2^{\frac{n+2}{2}}$. We then write this amplitude distribution as $[0^{\frac{2}{3}(2^n-1)},2^{\frac{1}{3}(2^n-1)}]$. 
\end{example}

For odd $n$, the amplitude distribution of quadratic APN functions $F \colon \F_2^n \rightarrow \F_2^n$ has been determined: all non-trivial components have amplitude $2^{\frac{n+1}{2}}$~\cite{carlet1998codes}. The $n$ even case, however, is still wide open, and essentially the only known constraint on the amplitudes comes from the well-known characterization of APN functions via the fourth moment of the Walsh transform, which, for quadratic APN functions, is equivalent to:
\begin{proposition}[{\cite[Proposition 9]{carlet2015boolean}}] \label{prop:fourth}
    Let $F \colon \F_2^n \rightarrow \F_2^n$ be a quadratic APN function, and let $2^{\frac{n+l_b}{2}}$ be the amplitude of component $F_b$. Then \[\sum_{b\in \F_2^n\setminus\{0\}}2^{l_b}=2(2^n-1).\]
\end{proposition}
 
\subsection{Our contributions and structure of the paper}
This paper gives new conditions and connections for the amplitude distribution of quadratic APN functions $F \colon \F_2^n \rightarrow \F_2^n$, where $n$ is even. In Section~\ref{s:partitions}, we show that the amplitude distribution of $F$ gives rise to a vector space partition of $\F_2^n$ (Theorem~\ref{thm:dim=amplitude}). Using the established theory of vector space partitions, we are then able to give additional constraints on the amplitude distributions. For example, we are able to prove that $F$ can have only one component function with amplitude larger than $2^{3n/4}$ (Theorem~\ref{thm:onebig}), and derive bounds on the number of bent components of a quadratic APN function (Propositions~\ref{prop:LB1} and~\ref{prop:LB2}). We then provide an analysis of possible amplitude distributions of quadratic APN functions in low dimensions (up to $n=10$) and list all possible relevant vector space partitions.

In Section~\ref{s:blocking}, we show that the non-trivial non-bent component functions of a quadratic (not necessarily APN) function form a special blocking set in a suitable projective space. This leads again to additional constraints. In particular, we prove new results on the number of bent component functions of (general) quadratic functions (Theorem~\ref{thm:generalbentbound}). For the case of quadratic APN functions, our results yield the first nontrivial, general upper bound on the number of bent components (Theorem~\ref{thm:APN bent upper bound}).

In Section~\ref{s:combine}, we combine our results from the previous sections, excluding more possible amplitude distributions for quadratic APN functions. We show, for instance, that every quadratic APN function on $\F_2^n$ has a component function with amplitude less than $2^{3n/4}$ (Corollary~\ref{cor:small}). For the case $n=8$, we show that there are at most 18 different amplitude distributions (and thus extended Walsh spectra) for quadratic APN functions (Theorem~\ref{thm:dim8final}). 

In Section~\ref{s:equivalence}, we discuss how the blocking sets and vector space partitions associated with quadratic APN functions behave under equivalence. In particular, we prove that if a quadratic APN function is CCZ-equivalent to a permutation, it cannot have few non-bent components (Corollary~\ref{cor:CCZ}). We conclude the paper with a list of open problems.

\section{Vector Space Partitions in Relation to Amplitudes} \label{s:partitions}
 Throughout this section, let $n$ be even and $F\colon \F_2^n \rightarrow \F_2^n$ be a quadratic APN function.
Define the difference function $D_{F,a}(x) = F(x)+F(x+a)$. Since $F$ is APN and quadratic, for every nonzero $a$, the image set of $D_{F,a}$ is an affine hyperplane of $\F_2^n$.

We denote the affine hyperplanes by 
\[
    H_b=\{x\in \F_2^n: \langle b, x\rangle = 0\}, \;\; \overline{H_b}=\{x\in \F_2^n: \langle b, x\rangle = 1\}.\]
 Clearly, $H_b$ and $\overline{H_b}$ are complements in $\F_2^n$. We further define the following sets
 \begin{align*}
    T_b&=\{a \in \F_2^n: \image(D_{F,a})= H_b\} \cup \{0\},\;\; \overline{T_b}=\{a\in \F_2^n: \image(D_{F,a}) = \overline{H_b}\},\\
    V_b&=  T_b \cup \overline{T_b}.
 \end{align*}


We have the following simple lemma about the structure and size of the sets $T_b$ and $V_b$. These results are essentially already given by Kyureghyan~\cite[Proposition 1]{kyureghyan2007crooked} with slightly different notation; we include the proof for completeness.

\begin{lemma} \label{lem:vs}
\leavevmode
    \begin{enumerate}[(i)]
        \item $T_b$ is a vector space over $\F_2$.
        \item If $\overline{T_b}$ is nonempty, then it is an affine space over $\F_2$  of the form $\overline{T_b}=c+T_b$ for some $c \notin T_b$.
        \item $V_b=T_b\cup\overline{T_b}$ is a vector space over $\F_2$ of dimension $\dim(T_b)$ (if $\overline{T_b}=\emptyset$), or $\dim(T_b)+1$ (if $\overline{T_b}\neq \emptyset$).
    \end{enumerate}
\end{lemma}
\begin{proof}
\leavevmode
    \begin{enumerate}[(i)]
        \item $0 \in T_b$ for all $b$ by definition. Suppose $a_1, a_2 \in T_b$. Then for all $x \in \mathbb{F}_2^n$, \begin{equation*}F(x+a_1+a_2)+F(x)=(F(x+a_1+a_2)+F(x+a_1))+(F(x+a_1)+F(x))\end{equation*}
        Since $F(x+a_2)+F(x) \in H_b$ for all $x \in \F_2^n$, we have  $F(x+a_1+a_2)+F(x+a_1) \in H_b$ for all $x \in \F_2^n$. 
        Since $F(x+a_1)+F(x) \in H_b$ for all $x \in \F_2^n$, and $H_b$ is a vector space, it follows that $F(x+a_1+a_2)+F(x) \in H_b$ for all $x \in \F_2^n$, that is, $a_1+a_2 \in T_b$, and $T_b$ is a vector space over $\F_2$ as claimed.
        \item Assume $\overline{T_b} \neq \emptyset$. Suppose that we have some $c \in \overline{T_b}, a \in T_b$. We again have
        \begin{equation*} F(x+a+c)+F(x) = (F(x+a+c)+F(x+c))+(F(x+c)+F(x)).
        \end{equation*}
        By the assumptions on $a$ and $c$, we have $\langle b, F(x+a+c)+F(x+c) \rangle = 0$ and $\langle b, F(x+c)+F(x) \rangle = 1$. We conclude that $\langle b, F(x+a+c)+F(x) \rangle=1$ for all $x$, meaning that $a+c \in \overline{T_b}$, so $T_b+c \subseteq \overline{T_b}$. 
        
        Conversely, assume $a\in \overline{T_b}$. 
        Then  $\langle b, F(x+a+c)+F(x+c) \rangle = 1$ and $\langle b, F(x+c)+F(x) \rangle = 1$, so $\langle b, F(x+a+c)+F(x) \rangle=0$ for all $x$, meaning that $a+c \in T_b$, so $\overline{T_b}+c \subseteq T_b$.

        \item As proved in the previous item, $\overline{T_b}$ is either empty or a translate of $T_b$, i.e., $\overline{T_b} = c+T_b$ for some $c \notin T_b$. In that case, $V_b=T_b \cup \overline{T_b}$ is a vector space spanned by $c$ and $T_b$ and thus has dimension $\dim(T_b)+1$.
        \end{enumerate}
\end{proof} 

We will establish that the sets $V_b$ for $b \in \F_2^n\setminus \{0\}$ form a vector space partition of $\F_2^n$. We recall the definition of a (finite) vector space partition. 

\begin{definition} Let $\F_q^n$ be a vector space of dimension $n$ over a finite field $\mathbb{F}_q$. A vector space partition is a collection of subspaces $\mathcal{V}$ = \{$V_1,...,V_t$\} of  $\F_q^n$, such that every nonzero vector is contained in a unique member of $V_1,...,V_t$, i.e., \[
\F_q^n = \bigcup_{i=1}^t V_i \quad \text{and} \quad V_i \cap V_j = \{0\} \text{ for } i \ne j.
\] 
We say $\mathcal{V}$ is of type $[d_1^{n_1},d_2^{n_2},\dots,d_k^{n_k}]$ if for all $i \in \{1,\dots,k\}$, the partition $\mathcal{V}$ contains exactly $n_i$ vector spaces of dimension $d_i$. For uniqueness, we set $1\leq d_1<d_2<\dots$. 
 \end{definition}

The following theorem shows that the spaces $V_b$ are indeed a vector space partition of $\F_2^n$.

\begin{theorem}
   The set $\{V_b\colon b \in \F_2^n\setminus \{0\}\}$ is a vector space partition of $\mathbb{F}_2^n$. 
\end{theorem}

\begin{proof}
    $V_{b_1} \cap V_{b_2}=\{0\}$ follows immediately from the definition of the $V_b$. Since $F$ is a quadratic APN function, $\image(D_{F,a})$ is an affine hyperplane for any non-zero $a \in \F_2^n$, so $a \in T_b$ or $a \in \overline{T_b}$ for some nonzero $b$. In particular, $a \in V_b$ for some nonzero $b$.
\end{proof}
We can thus associate to every quadratic APN function $F$ a uniquely defined vector space partition $\mathcal{V}_F=\{V_b \colon b\ \in \F_2^n\setminus \{0\}, \dim(V_b)\geq 1\}$. Note that $|\mathcal{V}_F|$ can  be less than $2^n-1$ if $V_b=\{0\}$ for some non-zero $b$, and in fact that necessarily happens, as we will see later. 

As the following theorem shows, the dimension of the individual vector space $V_b\in \mathcal{V}_F$ is directly linked to the amplitude of $F_b$.

The proof technique of the following proposition is inspired  by a proof by Kyureghyan~\cite[Proof of Theorem 1]{kyureghyan2007crooked}, even though the goal and outcome in that paper is substantially different from the goal here.
\begin{theorem}
    
\label{thm:dim=amplitude}
    Let $n$ be even, $F\colon \F_2^n \rightarrow \F_2^n$ be a quadratic APN function, and $\mathcal{V}_F=\{V_b \colon b \in \F_2^n\setminus\{0\}, \dim(V_b)\geq 1\}\}$. Then $F_b$ has amplitude $2^{\frac{n+\dim(V_b)}{2}}$. 
    
    In particular, if $F$ has amplitude distribution $[0^{k_0},2^{k_2},\dots,(n-2)^{k_{n-2}},n^{k_{n}}]$, then $\mathcal{V}_F$ is of type $[2^{k_2},\dots,(n-2)^{k_{n-2}},n^{k_{n}}]$.
\end{theorem}

\begin{proof}
    We calculate $(W_F(b,c))^2$. 
    We have 
        \begin{align*}
        (W_F(b,c))^2&=\left(\sum_{x \in \F_2^n}(-1)^{\langle b, F(x)\rangle + \langle c, x\rangle}\right)\cdot \left(\sum_{y \in \F_2^n}(-1)^{\langle b, F(y)\rangle + \langle c, y\rangle}\right) \\
        &=\sum_{x,a \in\mathbb{F}_2^n}(-1)^{\langle b, F(x)+F(x+a)\rangle + \langle c, a\rangle} \\ &= \sum_{a \in\mathbb{F}_2^n}(-1)^{\langle c, a\rangle}\sum_{x \in\mathbb{F}_2^n}(-1)^{\langle b, F(x)+F(x+a)\rangle}.
    \end{align*}
    Observe that for a fixed $a$, the equation $\langle b, F(x)+F(x+a)\rangle=0$ is either satisfied by all $x \in \F_{2}^n$ (if $a \in T_b$), by no $x \in  \F_{2}^n$ (if $a \in \overline{T_b}$), or by exactly $2^{n-1}$ distinct  $x \in \F_{2}^n$ (if $a \notin V_b$). In particular, we have 
    \[\sum_{x \in\mathbb{F}_2^n}(-1)^{\langle b, F(x)+F(x+a)\rangle}=\begin{cases}
        2^n, & \text{ if }  a \in T_b \\
        -2^n, & \text{ if }  a \in \overline{T_b}\\
        0, & \text{ else.} 
    \end{cases}\]
    We now distinguish two cases:
    \\ \textbf{Case 1:} $\overline{T_b}\neq\emptyset$. Then, $\overline{T_b} = v+T_b$ for some $v \notin T_b$ by Lemma~\ref{lem:vs}(ii).     
    So,
    \begin{align*}
        (W_F(b,c))^2&= 2^n\left(\sum_{a \in T_b}(-1)^{\langle c, a\rangle}-\sum_{a \in\overline{T_b}}(-1)^{\langle c, a\rangle}\right) \\ &= 2^n\left(\sum_{a \in T_b}\left((-1)^{\langle c, a\rangle}-(-1)^{\langle c, v+a\rangle}\right)\right)\\ &= 2^n\left(\sum_{a \in T_b}\left((-1)^{\langle c, a\rangle}-(-1)^{\langle c, v\rangle}(-1)^{\langle c, a\rangle}\right)\right) \\ &= \begin{cases}
            0, & \text{if } (-1)^{\langle c,v\rangle}=1 \text{ or } c \notin T_b^\perp, \\ 2^n\cdot 2\left|T_b\right|=2^n\left|V_b\right|, & \text{else}.
        \end{cases}
    \end{align*}
    \textbf{Case 2:} $\overline{T_b}=\emptyset$.
    Then, \begin{equation*} \begin{split}
        (W_F(b,c))^2&=2^n\cdot\sum_{a\in T_b} (-1)^{\langle c,a\rangle} \\ &= 2^n\left(\left|T_b\cap H_c\right|-\left(\left|T_b\right|-\left|T_b\cap H_c\right|\right)\right) \\ &= 2^n \begin{cases}
            \left|T_b\right|, & \text{if } T_b \subseteq H_c, \\ 0, & \text{if } T_b \nsubseteq H_c.
        \end{cases}
    \end{split} 
    \end{equation*}
    The claim follows.
\end{proof}

\begin{remark}
    The assumption that $n$ is even is not essential in the proof of Theorem~\ref{thm:dim=amplitude}, and a similar result can be derived in the $n$ odd case. However, it is known that for odd $n$, all quadratic APN functions are almost bent, which means all non-trivial component functions have amplitude $2^{\frac{n+1}{2}}$. The corresponding vector space partition is then of type $[1^{2^n-1}]$, i.e., a trivial decomposition of $\F_2^n$ into one dimensional subspaces.
\end{remark}

We state some immediate corollaries from this main result. 

For the rest of the paper, we denote by $N_F$ the set of non-trivial non-bent component functions of $F$, i.e., 
\[N_F = \{b \in \F_2^{n}\setminus \{0\} \colon F_b \text{ is not bent}\}.\]
Note that by Theorem~\ref{thm:dim=amplitude},  $b \in N_F$ if and only if $V_b \neq\{0\}$. Recalling that we discard zero-dimensional spaces in the vector space partition $\mathcal{V}_F$, we get:
\begin{corollary} \label{cor:nb_partition}
    Let $n$ be even, $F\colon \F_2^n \rightarrow \F_2^n$ be a quadratic APN function, and $\mathcal{V}_F=\{V_b \colon b \in \F_2^n\setminus\{0\},\dim(V_b)\geq 1\}$ be its associated vector space partition of $\F_2^n$. Then, 
    $|N_F|=|\mathcal{V}_{F}|$. 
\end{corollary}

Since amplitudes are necessarily integers, Theorem~\ref{thm:dim=amplitude} immediately yields:
\begin{corollary} \label{cor:even}
    Let $n$ be even and $F\colon \F_2^n \rightarrow \F_2^n$ be a quadratic APN function with associated vector space partition $\mathcal{V}_F=\{V_b \colon b \in \F_2^n\setminus\{0\},\dim(V_b)\geq 1\}$. Then $\dim(V_b)$ is even for every $V_b \in \mathcal{V}_F$.
\end{corollary}

\subsection{Conditions on Vector Space Partitions and their Consequences to Amplitude Distributions of APN Functions}

We state some conditions on amplitude distributions that follow from the connection to vector space partitions in Theorem~\ref{thm:dim=amplitude}. For a survey on vector space partitions, we refer to~\cite{heden2012survey}.

We start with the fundamental conditions that follow from counting arguments and the dimension formula.
\begin{theorem} \label{thm:conditions}
    Let $n$ be even and $F \colon \F_{2}^n \rightarrow \F_{2}^n$ be a quadratic APN function with amplitude distribution $[0^{k_0},2^{k_2},\dots,n^{k_{n}}]$. Then 
    \begin{enumerate}[(i)]
        \item $\displaystyle\sum_{i=1}^{n/2} k_{2i}(2^{2i}-1)=2^n-1$. (packing condition).
        \item If $i>\frac{n}{2}$, then $k_i\leq1$. (dimension bound I).
        \item If $k_i, k_j\geq 1$, $i\neq j$, then $i+j\leq n$. (dimension bound II).
    \end{enumerate}
\end{theorem}
\begin{proof}
     By Theorem~\ref{thm:dim=amplitude}, the vector space partition $\mathcal{V}_F$ of  $\F_2^n$  associated with $F$ is of type $[2^{k_2},\dots,n^{k_{n}}]$. Item (i) now follows from basic counting since all elements in $\F_2^n \setminus \{0\}$ are contained in exactly one vector space in $\mathcal{V}_F$. Items (ii) and (iii) follow from the dimension formula that implies that for $V_1,V_2 \in \mathcal{V}_F$, $V_1\neq V_2$, we necessarily have $\dim(V_1)+\dim(V_2) \leq n$.
\end{proof}

\begin{remark}
    It turns out that the packing condition given in Theorem~\ref{thm:conditions} is equivalent to the condition from the fourth moment of the Walsh transform in Proposition~\ref{prop:fourth}. Recall that $2^{\frac{n+l_b}{2}}$ is the amplitude of the component $F_b$. Subtracting one from each summand in Proposition~\ref{prop:fourth} yields,
    \[\sum_{b \in \F_2^n \setminus \{0\}}(2^{l_b}-1) = 2^n-1.\]
    Each component with amplitude $2^{\frac{n+i}{2}}$ corresponds to a vector space of dimension $i$, and we get precisely Condition (i) in Theorem~\ref{thm:conditions}. This correspondence is reversible, so the packing condition is equivalent to Proposition~\ref{prop:fourth}. In particular, the well known trivial upper bound on the number of non-bent component functions $|N_F|\leq (2^n-1)/3$ follows immediately from the packing condition since every vector space belonging to a non-bent component corresponds to a vector space containing at least $3$ non-zero vectors.
\end{remark}

Interestingly, parts (ii) and (iii) of Theorem~\ref{thm:conditions} imply that $F$ can only have one non-trivial component function with large amplitude:
\begin{theorem} \label{thm:onebig}
     Let $n$ be even and $F \colon \F_{2}^n \rightarrow \F_{2}^n$ be a quadratic APN function with linearity $L(F)=2^{\frac{n+l}{2}}$ with $l > n/2$. Then $F$ has exactly one component function with amplitude $2^{\frac{n+l}{2}}$, and all other components have amplitude at most $2^{\frac{2n-l}{2}}$.
     In particular, any quadratic APN function $F \colon \F_{2}^n \rightarrow \F_{2}^n$ has at most one component function with amplitude  larger than $2^{3n/4}$.
\end{theorem}
\begin{proof}
    $F$ has linearity $L(F)=2^{\frac{n+l}{2}}$, so it has at least one component function $F_b$ with amplitude $2^{\frac{n+l}{2}}$. Assume there is a second component function with the same amplitude. The amplitude distribution of $F$ is then $[0^{k_0},2^{k_2},\dots,l^{k_l}]$ with $k_{l} \geq 2$, contradicting Theorem~\ref{thm:conditions} (ii). So $F_b$ is the unique component function with amplitude $2^{\frac{n+l}{2}}$. 
    
    Assume now that there is another component function $F_{b'}$ with amplitude larger than $2^{\frac{2n-l}{2}}$. The amplitude distribution of $F$ thus satisfies $k_{l} =1$ and $k_{i} \geq 1$ for some $i > n-l$, contradicting Theorem~\ref{thm:conditions} (iii). 
\end{proof}

\begin{example} \label{ex:maxlinearity}
    It is well known (see, e.g.~\cite[Proposition 161]{Carlet2021_Book}) that an APN function $F\colon \F_{2}^n \rightarrow \F_2^n$ cannot have linearity $2^n$ as long as $n>2$. Quadratic APN functions with maximal linearity thus have linearity $2^{n-1}$. Examples of such functions for $n=8$ were found in~\cite{beierle2021new}. It was also shown using completely different techniques that the amplitude distribution of quadratic APN functions with linearity $2^{n-1}$ is $[0^{2^n-2^{n-2}-2},2^{2^{n-2}},(n-2)^{1}]$, see~\cite[Theorem 4]{beierle2022trims}. Using our tools, we can recover this result easily: A quadratic APN function with linearity $2^{n-1}$ has by Theorem~\ref{thm:onebig} only one component with amplitude $2^{n-1}$, and all other components have linearity $2^{\frac{n+k}{2}
    }$ with $k \in \{0,2\}$. The multiplicities for these amplitudes follow then easily  from the packing condition in Theorem~\ref{thm:conditions} (i).
\end{example}
We now mention some more elaborate bounds on vector space partitions. 

Let $\mathcal{V}$ be a vector space partition of type $[d_1^{n_1},d_2^{n_2},\dots,d_k^{n_k}]$. 
The \emph{tail} of $\mathcal{V}$ is the smallest $d_i$ such that $n_i \geq 1$, i.e., the smallest dimension of a subspace in $\mathcal{V}$. We then call $n_i$ the \emph{length} of the tail of $\mathcal{V}$.

We have the following result on tails of vector space partitions:
\begin{theorem}\cite{Heden2009}
\label{thm: Heden}
    Let $\mathcal{V}$ be a partition of $\F_2^n$, and let $k$ denote the length of the tail of $\mathcal{V}$, let $d_1$ denote the tail of $\mathcal{V}$, and let $d_2$ denote the dimension of the vector spaces of second lowest dimension. Then
    \begin{enumerate}[(i)]
        \item if $2^{d_2-d_1}$ does not divide $k$ and $d_2<2d_1$, then $k\geq2^{d_1}+1$.
        \item if $2^{d_2-d_1}$ does not divide $k$ and $d_2\geq 2d_1$, then either $d_1$ divides $d_2$ and $k=\frac{2^{d_2}-1}{2^{d_1}-1}$ or $k>2\cdot2^{d_2-d_1}$.
        \item if $2^{d_2-d_1}$ divides $k$ and $d_2<2d_1$, then $k\geq 2^{d_2}-2^{d_1}+2^{d_2-d_1}$.
        \item if $2^{d_2-d_1}$ divides $k$ and $d_2\geq2d_1$, then $k\geq2^{d_2}$.
    \end{enumerate}
\end{theorem}
A tightening of (iii) was provided in \cite{kurzconstructions}; however for our purposes, bounds (ii) and (iv) are the only relevant bounds, thus we exclude the tightened bound.

We translate these results on tails to amplitude distributions via Theorem~\ref{thm:dim=amplitude}.
\begin{theorem}
    Let $n$ be even and $F\colon \F_{2}^n \rightarrow \F_2^n$ be a quadratic APN function. Let $2^{\frac{n+d_1}{2}}$ be the smallest non-bent amplitude of $F$, $2^{\frac{n+d_2}{2}}$ the second smallest non-bent amplitude of $F$, and $k$ the number of component functions with amplitude $2^{\frac{n+d_1}{2}}$. Then:  \begin{enumerate}[(i)]
        \item if $2^{d_2-d_1}$ does not divide $k$ and $d_2<2d_1$, then $k\geq2^{d_1}+1$.
        \item if $2^{d_2-d_1}$ does not divide $k$ and $d_2\geq 2d_1$, then either $d_1$ divides $d_2$ and $k=\frac{2^{d_2}-1}{2^{d_1}-1}$ or $k>2\cdot2^{d_2-d_1}$.
        \item if $2^{d_2-d_1}$ divides $k$ and $d_2<2d_1$, then $k\geq 2^{d_2}-2^{d_1}+2^{d_2-d_1}$.
        \item if $2^{d_2-d_1}$ divides $k$ and $d_2\geq2d_1$, then $k\geq2^{d_2}$.
    \end{enumerate}
\end{theorem}

As mentioned in Corollary~\ref{cor:nb_partition}, the size of the set of non-trivial non-bent components $N_F$ of a quadratic APN function is exactly the size of the vector space partition $\mathcal{V}_F$. In particular, bounds on the size of vector space partitions yield immediate bounds for the size of $N_F$. The following two propositions, for instance, follow immediately from~\cite[Theorem 1]{heden2012extremal} via Theorem~\ref{thm:dim=amplitude}.

\begin{proposition}
\label{prop:LB1}
    Let $n$ be even, $F\colon \F_{2}^n \rightarrow \F_2^n$ be a quadratic APN function with linearity $2^{\frac{n+k}{2}}$ with $k \leq n/2$. Set $n=sk+r$ with integers $s,r$ and $0\leq r<k$. 
    If $k \mid n$, then 
    \[|N_F| \geq \frac{2^{n}-1}{2^k-1},\]
    and if $k \nmid n$, then
    \[|N_F| \geq 1+2^{\lceil \frac{k+r}{2}\rceil}+2^{k+r}\cdot \sum_{i=0}^{s-2}2^{ik}. \]
\end{proposition}

\begin{proposition}
    Let $n$ be even, $F\colon \F_{2}^n \rightarrow \F_2^n$ be a quadratic APN function  and assume all components have amplitude at least $2^{\frac{n+k}{2}}$. Set $n=sk+r$ with integers $s,r$ and $0\leq r<k$. 
Then
    \[|N_F| \leq 1+2^{k+r}\cdot \sum_{i=0}^{s-2}2^{ik}.\]
\end{proposition}

\begin{proposition}
\label{prop:LB2}
    Let $n$ be even, $F\colon \F_{2}^n \rightarrow \F_2^n$ be a quadratic APN function  with linearity $2^{\frac{n+k}{2}}$, and assume $F$ has a component function with amplitude $2^{\frac{n+\ell}{2}}$ with $2\leq \ell<n-k$. Then $|N_F| \geq 2^k+2^\ell+1.$ 
\end{proposition}
\begin{proof}
    This follows from~\cite[Proposition 7]{lehmann2012some} together with Theorem~\ref{thm:dim=amplitude}.
\end{proof}

\subsection{Explicit constructions of vector space partitions}
In this section, we review some explicit constructions of vector space partitions. We only cover some of the most basic examples. We invite the reader to refer to, for instance~\cite{heden2012survey,lehmann2012some}, for some explicit constructions of vector space partitions.

The most famous vector space partitions are partitions into subspaces of the same dimension $t$, in finite geometry these partitions correspond exactly to $t$-spreads which are well known and studied (see e.g.~\cite{dembowski1968finite}).

\begin{construction} \label{constr:spread}
    A vector space partition of $\F_2^n$ that consists of vector spaces with the same dimension $t$ exists if and only if $t \mid n$. Such a partition is of type $[t^{\frac{2^n-1}{2^t-1}}]$ and is called a $t$-spread.
\end{construction}
For instance, the vector space partition associated with APN functions with the classical Walsh transform is exactly a $2$-spread, which exists for all even $n$. 

The next construction is due to Bu~\cite[Lemma 4]{bu1980partitions}.

\begin{construction} \label{constr:bu}
Identify $\F_2^n$ with $ \F_{2^{n-s}} \times \F_{2^{s}}$ where $s \mid (n-s)$.
Let $U_\infty=\F_{2^{n-s}} \times \{0\} \leq \F_{2^{n-s}} \times \F_{2^{s}}$, and define $U_\alpha=\{(\alpha x,x) : x \in \F_{2^s}\}$ for all $\alpha \in \F_{2^{n-s}}$. Then $\mathcal{V}=\{U_\infty\} \cup\{U_\alpha \colon \alpha \in \F_{2^s}\}$ is a vector space partition of type $[s^{2^{n-s}},(n-s)^{1}]$.
\end{construction}

Note that for $s=2$ this partition type corresponds precisely to quadratic APN functions with maximum linearity $2^{n-1}$, see Example~\ref{ex:maxlinearity}.

In~\cite[Example 6]{kurzconstructions}, Kurz provides a construction of the same type of vector space partition under the relaxed condition that $n=2s+r$ with $s \geq 2$ and $r\geq 1$. We omit the details and refer to~\cite[Example 6]{kurzconstructions} for the intricacies of the construction.
\begin{construction} \label{constr:kurz}
    Let $n=2s+r$ with $s\geq 2$ and $r\geq 1$. Then there exists a vector space partition of $\F_2^n$ of type $[s^{2^{n-s}},(n-s)^1]$.
\end{construction}



\subsection{Analysis for small dimensions}
We now consider vector space partition types of $\F_2^n$ for small even values of $n$. Via Theorem~\ref{thm:dim=amplitude}, these partition types correspond to possible amplitude distributions of quadratic APN functions. With Corollary~\ref{cor:even} in mind, we only consider vector space partitions where every subspace has even dimension.
\subsubsection{Dimension 6} \label{s:dim6}
There are only two possible vector space partition types: By the dimension bound (Theorem~\ref{thm:conditions}), if there is a $V \in \mathcal{V}$ with $\dim(V)=4$, then all other subspaces in $\mathcal{V}$ have dimension $2$. Hence, by the packing condition, the only two possible partition types are $[2^{21}]$ and $[2^{16},4^1]$. By Constructions~\ref{constr:spread} and~\ref{constr:bu}, both of these partition types also exist. Quadratic APN functions over $\F_2^6$ are completely classified~\cite[Table 4]{pott2016almost}, and there exist APN functions with amplitude distributions $[0^{42},2^{21}]$ and $[0^{46},2^{16},4^1]$ corresponding to the two partition types. So in dimension 6, every admissible partition type is realized by a quadratic APN function with corresponding amplitude distribution.

\begin{remark}
    Just from the classical result on the fourth moment of the Walsh transform (Proposition~\ref{prop:fourth}), other amplitude distributions seem possible, e.g. $[0^{50},2^{11},4^2]$, which however contradict the new result in Theorem~\ref{thm:onebig}. Our results in this paper thus illuminate why there are only two different Walsh spectra among quadratic APN functions on $\F_2^6$.
\end{remark}

\subsubsection{Dimension 8} \label{s:dim8}
The complete classification of vector space partitions types in $\F_2^8$ (excluding "degenerate" partitions containing one-dimensional spaces) was achieved in~\cite{el2010partitions} with the help of a computer. Since we are only interested in vector space partitions where every space in the partition has even dimension, we can give a computer-free proof for this case:

\begin{theorem} \label{thm:dim8}
    Let $\mathcal{V}$ be a vector space partition of $\F_2^8$ such that $\dim(V)$ is even for all $V \in \mathcal{V}$. Then $\mathcal{V}$ has one of the following types:
    \begin{itemize}
        \item $[8^1]$,
        \item $[2^{64},6^1]$,
        \item $[2^{5i},4^{17-i}]$ for $0 \leq i \leq 17$.
    \end{itemize}
    There exist vector space partitions of $\F_2^8$ for all these types.
\end{theorem}
\begin{proof}
    The case $[8^1]$ is trivial. The dimension bound along with the packing condition (see Theorem~\ref{thm:conditions}) imply that if there is a $V \in \mathcal{V}$ with $\dim(V) = 6$, then all other subspaces of $\mathcal{V}$ must have dimension 2 and the packing conditions shows that $[2^{64}, 6^1]$ is the only possibility.  The existence of such a partition is guaranteed by Construction~\ref{constr:bu}. The packing condition further implies that if $\dim(V) \leq 4$ for all $v \in \mathcal{V}$ then the type of $\mathcal{V}$ is necessarily $[2^{5i}, 4^{17-i}]$ for some $\ 0\leq i \leq 17.$ The partition of type $[4^{17}]$ is exactly a spread and exists by Construction~\ref{constr:spread}. Each vector space of dimension $4$ can further be partitioned into a $2$-spread, i.e., 5 subspaces of dimension 2 each. This shows that vector space partitions of types $[2^{5i},4^{17-i}]$ exist for all $0 \leq i \leq 17$.
\end{proof}

\subsubsection{Dimension 10} \label{s:dim10}
We now consider vector space partition of $\F_2^{10}$ in which all subspaces have even dimension. There is some intersection with previous work, e.g. in~\cite{seelinger}, the authors consider subspace partitions of the type $[2^as^b]$. We give a self-contained proof for all cases that are relevant to our question.
\begin{theorem} \label{thm:dim10}
    Let $\mathcal V$ be a vector space partition of $\F_2^{10}$ such that $\dim(V)$ is even for all $V \in \mathcal{V}$. Then $\mathcal V$ has one of the following types:
    \begin{itemize}
        \item $[10^1]$,
        \item $[2^{256},8^1]$,
        \item $[2^{320-5i},4^{i},6^1]$ for $0\le i\le 64$,
        \item $[2^{341-5i},4^i]$ for $0\le i\le 65.$
    \end{itemize}
    There exist examples of all these vector space partition types.
\end{theorem}

\begin{proof}
    The $[10^1]$ case is again trivial. If there is a $V \in \mathcal{V}$ with $\dim(V)=8$ then (by the dimension bound and packing condition) the vector space partition is necessarily of type $[2^{256},8^1]$, and Construction~\ref{constr:bu} gives again an example. 
    
    Assume now that there is vector space $V \in \mathcal{V}$ with $\dim(V)=6$. By the dimension bound, there cannot be another vector space in $\mathcal{V}$ with dimension $\geq 6$. A simple calculation shows that the packing condition only allows vector space partitions of the form $[2^{320-5i},4^{i},6^1]$ for $0\le i\le 64$. We now show that all these partitions in fact occur. A partition of type $[4^{64},6^1]$ can be constructed using Construction~\ref{constr:kurz} with $s=4$, $r=2$. 
    Each 4-dimensional subspace can itself be partitioned into a 2-spread, we therefore obtain examples of vector space partitions of type $[2^{5j},4^{64-j}, 6^1]$ where $0 \leq j \leq 64$, or, equivalently, of type $[2^{320-5i},4^{i}, 6^1]$ where $0 \leq i \leq 64$. 
    
    We now consider vector space partitions where all subspaces have dimension 2 or 4. By the packing condition, the possible vector space partition types are $[2^{341-5i},4^{i}]$ for $0\le i\le 68.$ However, by the tail condition in Theorem~\ref{thm: Heden} (ii), the partition types for $i\in \{67,68\}$ do not occur. Similarly, the partition type for $i=66$ cannot occur by~\cite[Lemma 16]{nuastase2026second}. 
    
    We now prove the existence of vector space partitions of type $[2^{341-5i},4^{i}]$ for $0\le i\le 65.$ Consider a vector space partition of type $[4^{64},6^1]$. Using Construction~\ref{constr:bu}, we can partition the $6$-dimensional space in this partition further into one $4$-dimensional space and 16 spaces of dimension 2. We arrive at a partition of type $[2^{16},4^{65}]$. We can now again partition each $4$-dimensional space in this partition into a $2$-spread, i.e., 5 subspaces of dimension $2$. We thus obtain examples of vector space partitions of type $[2^{16 +5j},4^{65-j}]$ for $0 \leq j \leq 65$, which is the same as  $[2^{341-5i},4^{i}]$ for $0\le i\le 65.$
\end{proof}

Theorem~\ref{thm:dim10} gives a complete answer for the relevant vector spaces for amplitude distributions of quadratic APN functions in dimension 10. 

\section{Blocking sets associated to quadratic functions}
\label{s:blocking}
In this section, we study the set of non-trivial non-bent components of a quadratic function $F\colon \F_2^n \rightarrow \F_2^m$ further. Recall that we denote this set by $N_F$, i.e., $N_F=\{b \in \F_2^m\setminus \{0\} \colon F_b \text{ is not bent}\}$.
\subsection{General case}
We first consider the non-bent components of a general quadratic function $F \colon \mathbb{F}_2^n \rightarrow \mathbb{F}_2^m$. 
\begin{proposition} \label{prop:odd}
    Let $n$ be even, and let $F \colon \mathbb{F}_2^n \rightarrow \mathbb{F}_2^m$ be quadratic. If $W\leq \F_2^m$ is a vector space of dimension at least $\frac{n}{2}+1$, then $|W\cap N_F |$ is odd. 
\end{proposition}

\begin{proof}
    By orthogonality of characters, \begin{equation*}
    \begin{split}
        \sum_{b\in W} W_F(b,0)&=\sum_{x\in\mathbb{F}_2^n}\sum_{b\in W} (-1)^{\langle b, F(x)\rangle} \\&= |W|\cdot\left|\{x:F(x)\in W^{\perp}\}\right|.
    \end{split}
\end{equation*} Hence, the left-hand side is divisible by $|W|$, and in particular by $2^{\frac{n}{2}+1}$. 

Now let 
\[P=|\{b \in W \colon W_F(b,0)=2^\frac{n}{2}\}|, \qquad B=|\{b \in W \colon |W_F(b,0)|=2^\frac{n}{2}\}|.\]
Then, 
\begin{align*}
    \sum_{b\in W} W_F(b,0)&=2^n+2^\frac{n}{2}\cdot P-2^\frac{n}{2}\cdot (B-P) + 2^{\frac{n}{2}+1}\cdot C\\&= 2^n+2^\frac{n}{2}\cdot (2P-B)+2^{\frac{n}{2}+1}\cdot C
\end{align*}
for some integer $C$. Now since the left-hand side is divisible by $2^{\frac{n}{2}+1}$, we have that $B$ is necessarily even. This means the number of non-trivial non-bent components in $W$ is odd.
\end{proof}
In the language of finite geometry, this means that $N_F$ is a certain kind of blocking set. We recall the terminology.
\begin{definition}
    Let $\PG(m,q)$ be the projective space of dimension $m$ over the finite field $\F_q$. A subset $B\subseteq \PG(m,q)$ is called a blocking set with respect to $k$-spaces, if $B$ intersects every $k$-space in at least one element.
\end{definition}
Recall that we can identify $\PG(m-1,2)$ naturally with $\F_2^m\setminus\{0\}$. Under this correspondence, a $k$-space in $\PG(m-1,2)$ is exactly $V\setminus \{0\}$, for some $(k+1)$-dimensional subspace $V$ of $\F_2^m$. Under this correspondence, Proposition~\ref{prop:odd} immediately yields the following:
\begin{corollary} \label{cor:blockingset}
    Let $n$ be even and $F \colon \mathbb{F}_2^n \rightarrow \mathbb{F}_2^m$ be quadratic. Then $N_F$ is a blocking set in $\PG(m-1,2)$ with respect to $(n/2)$-spaces, moreover $|W \cap N_F|$ is odd for each $(n/2)$-space $W$.
\end{corollary}
Note that it was already observed in~\cite{pott2017maximum} that $N_F$ is a blocking set in $\PG(m-1,2)$ with respect to $(n/2)$-spaces (this is indeed true even for non-quadratic functions). However, the additional information that the intersections have to be odd offers a new crucial insight that leads to sharper results as we will see in this section. Let us first recall some basic facts on blocking sets and their consequences for the number of bent component functions of quadratic functions.

\begin{theorem}\cite{bose1966characterization} \label{thm:boseburton}
    Let $B$ be a blocking set with respect to $k$-spaces in $\PG(m,2)$. Then $|B|\geq 2^{m-k+1}-1$, and equality holds if and only if $B$ is an $(m-k)$-space.
\end{theorem}
As observed in~\cite{pott2017maximum,zheng2020constructions}, this result implies that a function $F \colon \F_2^n\rightarrow \F_2^m$ has at most $2^m-2^{m-n/2}$ bent components; equivalently, $|N_F| \geq 2^{m-n/2}-1$. Functions satisfying this bound with equality are called \emph{MNBC functions} (for \emph{maximum number of bent components}). 

Blocking sets with respect to $k$-spaces in $\PG(m,2)$ that contain an $(m-k)$-space are called \textit{trivial} blocking sets. A blocking set $B$ such that no subset of $B$ is a blocking set is called a \textit{minimal} blocking set. 
The smallest non-trivial (and hence minimal) blocking sets in $\PG(m,2)$ with respect to $k$-spaces were found by Govaerts and Storme~\cite{govaerts2006classification}:
\begin{theorem}{\cite[Theorem 1.4. (3)]{govaerts2006classification}}\label{thm:govaerts}
    In $\PG(m,2)$ for $m \geq 3$, the smallest non-trivial blocking sets with respect to $k$-spaces $(1 \leq k \leq m-2)$ have size $2^{m-k+1}+2^{m-k-1}+2^{m-k-2}-1$. 
\end{theorem}

The following result shows that small blocking sets with the property that all relevant intersections are odd are necessarily minimal. The proof is inspired by~\cite[Lemma 3.1]{szHonyi2001small}. We will use the well known fact that each $t$-space in $\PG(m,2)$ is contained in exactly $2^{m-t}-1$ distinct $(t+1)$-spaces.
\begin{proposition} \label{prop:szonyi}
    Let $B$ be a blocking set of $\PG(m,2)$ with respect to $k$-spaces, such that $B$ intersects every $k$-space in an odd amount of points. If $|B|\leq 2^{m-k+2}-2$ then $B$ is a minimal blocking set.
\end{proposition}
\begin{proof}
    Assume that $B$ is not minimal, and let $P \in B$ such that $B\setminus \{P\}$ is still a blocking set with respect to $k$-spaces. In particular, every $k$-space containing $P$ must intersect $B$ in at least $3$ points. 
    
    Suppose that $V$ is a $(k-1)$-space such that $V\cap B=\{P\}$. There are then $2^{m-k+1}-1$ many $k$-spaces containing $V$, and each of them has to contain at least $2$ points of $B\setminus \{P\}$. Thus $|B| \geq 2\cdot (2^{m-k+1}-1)+1$, contradicting the assumption, and the claim follows. It thus remains to show that such a $V$ exists.

    We show by induction that there is a $t$-space $V_t$ such that  $V_t\cap B=\{P\}$ for each $0 \leq t \leq k-1$. For $t=0$, we can simply take $V_0$ to be the point $P$ itself. So assume by induction hypothesis that there exists a $(t-1)$-space $V_{t-1}$, such that $V_{t-1}\cap B=\{P\}$. There are again $2^{m-t+1}-1$ distinct $t$-spaces containing $V_{t-1}$. Since $t \leq k-1$, we have $2^{m-t+1}-1>|B|$, by our assumption on its size. So there has to be one $t$-space that does not contain an element of $B\setminus\{P\}$, proving our claim.
\end{proof}

We can combine Theorem~\ref{thm:govaerts} and Proposition~\ref{prop:szonyi} to get:
\begin{theorem} \label{thm:generalbentbound}
    Let $n$ be even with $n\leq 2m-6$ and $F \colon \mathbb{F}_2^n \rightarrow \mathbb{F}_2^m$ be quadratic. If $F$ is not an MNBC function then $|N_F| \geq 2^{m-n/2}+2^{m-n/2-2}+2^{m-n/2-3}-1$. In other words, $F$ has at most $2^{m}-2^{m-n/2}-2^{m-n/2-2}-2^{m-n/2-3}$ bent component functions. 
\end{theorem}
\begin{proof}
    Assume $|N_F|<2^{m-n/2}+2^{m-n/2-2}+2^{m-n/2-3}-1$. By Corollary~\ref{cor:blockingset}, $N_F$ is a blocking set in $\PG(m-1,2)$ with respect to $(n/2)$-spaces such that $N_F$ intersects all $(n/2)$-spaces in an odd amount of points. By Proposition~\ref{prop:szonyi}, $N_F$ has to be a minimal blocking set, which by Theorem~\ref{thm:govaerts} implies that it is trivial. But then $F$ is an MNBC function, contradicting the assumption.
\end{proof}

\subsection{The case of quadratic APN functions}
It is known that quadratic APN functions cannot be MNBC functions, as shown in~\cite[Theorem 2]{mesnager2019further}. This also follows immediately from the following observation made in~\cite{kolsch2024combinatorial}. 
\begin{theorem} \label{thm:1mod4}
      Let  $n$ be even and $F \colon \F_2^n \rightarrow \F_2^n$ be a quadratic APN function. Then $|N_F| \equiv 1 \pmod 4$.
\end{theorem}

However, no further constraints on the number of bent component functions of a quadratic APN function are known so far. Theorem~\ref{thm:generalbentbound} applied to APN functions thus provides the first non-trivial  upper bound on the number of bent component functions of a quadratic APN function.
\begin{theorem}
\label{thm:APN bent upper bound}
    Let $n\geq 6$ be even and $F \colon \F_2^n \rightarrow \F_2^n$ be a quadratic APN function. Then $|N_F| \geq 2^{n/2}+2^{n/2-2}+2^{n/2-3}-1$, where equality can occur only if $n=8$. In other words, the number of bent component functions of $F$ is at most $2^n-2^{n/2}-2^{n/2-2}-2^{n/2-3}$, and strictly less if $n \neq 8$.
\end{theorem}
\begin{proof}
    This follows from Theorem~\ref{thm:generalbentbound} by setting $n=m$ together with Theorem~\ref{thm:1mod4} under the observation that $2^{n/2}+2^{n/2-2}+2^{n/2-3}-1 \equiv 1 \pmod 4$ if and only if $n=8$.
\end{proof}

Let us end this section with an explicit example that shows that the set of non-trivial non-bent components can be a trivial blocking set, but is not necessarily so.
\begin{example}
    Let $n=2m$ be even and $F \colon \F_{2^n} \rightarrow \F_{2^n}$ be the APN function defined by $F(x)=x^3$. It is well known that in this case $N_F=\{x^3 \colon x \in \F_{2^n}^*\}$, i.e., the set of non-zero cubes. If $m$ is odd, then $3|(2^m+1)$. Every element in $\F_{2^m}$ is a $(2^m+1)$-st power in $\F_{2^n}$, in particular it is a cube. So $\F_{2^m}$ is an $m=n/2$-dimensional subspace entirely contained in $N_F \cup \{0\}$. In other words, $N_F$ is in this case a trivial blocking set.

    If $m$ is even, then by~\cite[Lemma 16]{golouglu2020almost}, the set of cubes does not contain an $m$-dimensional subset, so $N_F$ is non-trivial.
\end{example}
For some computational results on the maximal subspaces contained in $N_F\cup \{0\}$ for some quadratic APN functions in dimensions $\leq 10$, we also refer to~\cite[Table 2]{golouglu2021ccz}. These results also show that $N_F$ is sometimes a trivial and sometimes a nontrivial blocking set.

\begin{remark}
    We note that all results as well as proofs in this section apply not only to quadratic APN functions $F$, but to all \emph{plateaued} APN functions, since the only property of quadratic functions we have used in this section is that $|W_F(b,0)| \in \{0,2^{(n+k)/2}\}$ for some $k$. However, this is not the case for the results in Section~\ref{s:partitions}, where we used the fact that the differential sets of quadratic APN functions are affine hyperplanes, which is not necessarily true for plateaued APN functions. In fact, these functions are usually called \emph{crooked}, see~\cite{kyureghyan2007crooked}. For simplicity (and since no crooked, non-quadratic functions are known), we stick to the terminology \emph{quadratic} in this paper. 
\end{remark}

\section{Combining vector space partitions and blocking sets } \label{s:combine}

In Section~\ref{s:partitions}, we showed that every quadratic APN function is uniquely linked to a vector space partition of the ambient space, and we derived conditions on the amplitude distribution of such a function based on results on vector space partitions. In this section, we investigate the question: Which vector space partitions are never associated to any quadratic APN function? Currently, all known infinite families of quadratic APN functions on $\F_2^n$ with even $n$ have classical Walsh spectrum, which has an associated vector space partition of type $[2^{\frac{1}{3}(2^n-1)}]$.

Clearly, by Corollary~\ref{cor:even}, any vector space partition containing odd-dimensional subspaces cannot correspond to the amplitude distribution of any quadratic APN function. Furthermore, the trivial partition $[n^1]$ of $\F_2^n$ is also never associated to an APN function  as long as $n>2$. This can be seen in multiple ways. For instance, such an APN function would have $2^n-1$ bent components by Corollary~\ref{cor:nb_partition}, which is impossible by Theorem~\ref{thm:APN bent upper bound} for $n \geq 6$. Another way to show this is that by Theorem~\ref{thm:dim=amplitude}, such a function would have linearity $2^n$, which is impossible  if $n>2$~\cite[Proposition 161]{Carlet21}.

Using Theorem~\ref{thm:APN bent upper bound}, we can prove that the $(n/2)$-spread is never associated to an amplitude distribution. If $4 \mid n$, this is the first example of a non-trivial vector space partition with even dimensional subspaces whose associated amplitude distribution does not occur for a quadratic APN function on $\F_2^n$:
\begin{theorem} \label{thm:spread}
    Let $n\geq 6$ be even. Then vector space partitions of type
    $\left[(n/2)^{2^{\frac{n}{2}}+1}\right]$
    do not correspond to the amplitude distribution of any quadratic APN function.
\end{theorem}

\begin{proof}
If $n/2$ is odd, then the statement is trivially true because the vector space dimensions need to be even. Otherwise, the
 $(n/2)$-spread corresponds to an amplitude distribution $[0^{2^n-2^{n/2}-2},(n/2)^{2^{n/2}+1}]$, in particular, a quadratic APN function with this amplitude distribution would have $|N_F|=2^{n/2}+1$. By Theorem~\ref{thm:APN bent upper bound}, a quadratic APN function 
$F \colon \F_2^n \to \F_2^n$ has more than $2^{n/2}+1$ non-bent components.
\end{proof}

\begin{corollary} \label{cor:small}
    Let $n \geq 8$ be even and $F\colon \F_2^n \rightarrow \F_2^n$ be a quadratic APN function. Then $F$ has a non-bent component function with amplitude less than $2^{3n/4}$.
\end{corollary}
\begin{proof}
 By Theorem~\ref{thm:spread}, it is impossible that all non-bent component functions have amplitude $2^{3n/4}$, so we necessarily have one  non-trivial component function of $F$ with amplitude less than $2^{3n/4}$ or one with amplitude larger than $2^{3n/4}$. In the second case, Theorem~\ref{thm:onebig} guarantees that all other component functions have amplitude less than $2^{3n/4}$.
\end{proof}

Of course, we can apply Theorem~\ref{thm:APN bent upper bound} also to vector space partitions outside of the $(n/2)$-spread, as the following example shows. 
\begin{example}
  Consider a vector space partition $\mathcal{V}$ of type $[2^{16}, 4^1, 6^{2^6}]$ of $\F_2^{12}$. Clearly, such a vector space partition exists: take a $6$-spread of $\F_2^{12}$, and then use Construction~\ref{constr:bu} to partition one of the spaces of dimension 6 into a space of dimension 4 and 16 spaces of dimension 2. We have $|\mathcal{V}| = 2^6+16+1<2^6+2^4+2^3-1$, so by Corollary~\ref{cor:nb_partition} and Theorem~\ref{thm:APN bent upper bound} there is no quadratic APN function whose amplitude distribution corresponds to $\mathcal{V}$. 
\end{example}

\subsection{Comparison of lower bounds on \(|N_F|\)}

Recall the three lower bounds on $|N_F|$ established in Sections~\ref{s:partitions} and~\ref{s:blocking}.  
Let $n$ be even and let $F \colon \F_2^n \rightarrow \F_2^n$ be a quadratic APN function.

\begin{itemize}
  \item[(A)] \textbf{Proposition~\ref{prop:LB1}:}  
  If $F$ has linearity $2^{(n+k)/2}$ with $k \leq n/2$, and $n=sk+r$ where $0 \le r < k$, then
  \[
    |N_F|\ \geq\ 
    \begin{cases}
      \dfrac{2^n-1}{2^k-1}, & \text{if } k\mid n,\\[6pt]
      1+2^{\lceil (k+r)/2\rceil}+2^{k+r}\displaystyle\sum_{i=0}^{s-2}2^{ik}, & \text{if } k\nmid n.
    \end{cases}
  \]

  \item[(B)] \textbf{Theorem~\ref{thm:APN bent upper bound}:}  
  For every quadratic APN function on $\F_2^n$ with $n\geq 6$,
  \[
    |N_F|\ \geq\ 2^{n/2}+2^{\,n/2-2}+2^{\,n/2-3}-1.
  \]

  \item[(C)] \textbf{Proposition~\ref{prop:LB2}:}  
  If $F$ has linearity $2^{(n+k)/2}$ and a second non-bent component of amplitude $2^{(n+\ell)/2}$ with $\ell < n-k$, then
  \[
    |N_F|\ \geq\ 2^k + 2^\ell + 1.
  \]
\end{itemize}

It is natural to ask in which parameter ranges each bound provides the strongest estimate. The following proposition summarizes the comparison.

\begin{proposition} \label{prop:comp}
Let \(n \geq 6\) be even and let \(F \colon \F_2^n \rightarrow \F_2^n\) be a quadratic APN function with linearity \(2^{(n+k)/2}\).  
Write \(n = sk + r\) with \(0 \le r < k\). Then:
\begin{enumerate}[(i)]
  \item If \(k < n/2\), then bound \textnormal{(A)} yields the strongest lower bound, dominating both \textnormal{(B)} and \textnormal{(C)}.
  \item If \(k = n/2\), then bound \textnormal{(B)} is strictly stronger than \textnormal{(A)} and \textnormal{(C)}, except for $n=8$, where bound \textnormal{(B)} and \textnormal{(C)} coincide.
  \item If \(k > n/2\), then bound \textnormal{(C)} dominates; bound \textnormal{(A)} is not applicable in this regime.
\end{enumerate}
\end{proposition}

\begin{proof}
We compare the three bounds in the respective regimes.

\medskip\noindent\textbf{(i) \boldmath{$k < n/2$}.}
Since $k$ is even, we have \(n - k \geq n/2 + 2\).  
When \(k \mid n\), bound~(A) gives
\[
  \frac{2^n - 1}{2^k - 1}
  = 2^{n-k} + 2^{n-2k} + \cdots + 1
  \;>\; 2^{\,n/2 + 2}.
\]
In contrast, bound~(B) equals \(2^{n/2} + 2^{n/2 - 2} + 2^{n/2 - 3} - 1\), which is strictly smaller.  
Hence, (A) is exponentially stronger for all \(k < n/2\).  For \(k \nmid n\), the dominant term in (A) is again of order \(2^{\,n-k}\), and thus, the same conclusion follows.  
Finally, since (C) gives \(2^k + 2^\ell + 1\) with \(k < n/2\)  and $\ell \leq k$ it satisfies
\[
  2^k + 2^\ell + 1  < 2^{n/2},
\]
so (A) dominates both (B) and (C) whenever \(k < n/2\).

\medskip\noindent\textbf{(ii) \boldmath{$k = n/2$}.}
In this case, \(k \mid n\) and bound~(A) simplifies to
\[
  \frac{2^n - 1}{2^{n/2} - 1} = 2^{n/2} + 1.
\]
Bound~(B) gives
\[
  |N_F| \geq 2^{n/2} + 2^{\,n/2 - 2} + 2^{\,n/2 - 3} - 1 ,
\]
which is strictly greater than \(2^{n/2} + 1\) for all \(n \ge 6\).  
Bound~(C) leads at best to \(|N_F| \ge 2^{n/2} + 2^{\,n/2 - 2} + 1\),  which is again smaller than (B) if $n>8$.  
Therefore, (B) provides the sharpest lower bound when \(k = n/2\).

\medskip\noindent\textbf{(iii) \boldmath{$k > n/2$}.}
If \(k > n/2\), then necessarily \(k \ge n/2 + 2\).  
Bound~(C) gives \(|N_F| \ge 2^k + 2^\ell + 1 \ge 2^{n/2 + 2} + 1\), which already exceeds bound~(B).  
Moreover, Proposition~\ref{prop:LB1} does not apply in this case, so (A) is unavailable.  
Hence, (C) is the strongest bound in the range \(k > n/2\).

\medskip
This completes the comparison.
\end{proof}

\subsection{The cases of small dimensions}
As detailed in Section~\ref{s:dim6}, there are quadratic APN functions corresponding to both possible vector space partition types of $\F_2^6$. In larger dimensions, the situation is quite different.

\subsubsection{Dimension 8}
As shown in Theorem~\ref{thm:dim8}, the existing vector space partition types of $\F_2^8$ where all subspace dimensions are even are $[8^1]$, $[2^{64},6^1]$, and $[2^{5i},4^{17-i}]$ for $0 \leq i \leq 17$. As discussed before, the trivial partition $[8^1]$ can never be associated to a quadratic APN function and by Theorem~\ref{thm:spread}, the same is the case for the $4$-spread $[4^{17}]$. We conclude:

\begin{theorem} \label{thm:dim8final}
    Let $F \colon \F_2^8\rightarrow \F_2^8$ be a quadratic APN function. Then $F$ has one of the following amplitude distributions:
    \begin{itemize}
        \item $[0^{190},2^{64},6^1]$
        \item $[0^{238-4i},2^{5i},4^{17-i}]$ for $1 \leq i \leq 17$.
    \end{itemize}
\end{theorem}

To the best of our knowledge, the only amplitude distributions of quadratic APN functions on $\F_2^8$ that have so far been found are $[0^{190},2^{64},6^1]$ and $[0^{238-4i},2^{5i},4^{17-i}]$ for $13 \leq i \leq 17$,~\cite[Table 1]{beierle2021new}. It is still unclear if the other amplitude distributions in Theorem~\ref{thm:dim8final} occur.
\subsubsection{Dimension 10}
In dimension $10$, the bound from Theorem~\ref{thm:APN bent upper bound} does not exclude any vector space partitions we found in Theorem~\ref{thm:dim10} since $n/2=5$ is odd, see the comparison of bounds in Proposition~\ref{prop:comp}. We thus have:
    \begin{theorem} 
    Let $F \colon \F_2^{10}\rightarrow \F_2^{10}$ be a quadratic APN function. Then $F$ has one of the following amplitude distributions:
    \begin{itemize}
        \item $[0^{766},2^{256},8^1]$,
        \item $[0^{702+4i},2^{320-5i},4^{i},6^1]$ for $0\le i\le 64$,
        \item $[0^{682+4i},2^{341-5i},4^i]$ for $0\le i\le 65.$ 
    \end{itemize}
\end{theorem}

\section{The behavior of vector space partitions and blocking sets under equivalence} \label{s:equivalence}

We call two functions $F_1,F_2 \colon \F_2^n \rightarrow \F_2^m$ \emph{EA-equivalent} if there exist $A_1 \in \AGL(m,2)$, $A_2 \in \AGL(n,2)$, and an affine mapping $A_3 \colon \F_2^n \rightarrow \F_2^m$ such that $F_1 = A_1 \circ F_2 \circ A_2 + A_3$. It is well-known and elementary to verify that if $F_1$ is APN/quadratic then $F_2$ is as well. In this section, we now investigate how for EA-equivalent quadratic functions $F_1$, $F_2$ the associated blocking sets $N_{F_1}$ and $N_{F_2}$ are related. If $F_1,F_2$ are APN, we also consider how the associated vector space partitions are related.

By the fundamental theorem of projective geometry, the automorphism group of $\PG(m-1,q)$ is $P\Gamma L(m,q)$. In our case we have $P\Gamma L(m,2) \cong \GL(m,2)$. We therefore say that two point sets $B_1$, $B_2 \in \PG(m-1,2)$ are \emph{equivalent} if there is an $L \in \GL(m,2)$ such that $B_2 = L \circ B_1$. 

\begin{theorem}
    Let $F,G \colon \F_2^n \rightarrow \F_2^m$ be two quadratic functions that are EA-equivalent via $G = A_1 \circ F \circ A_2 + A_3$. Let further $A_1=L_1+c_1$ (i.e., $L_1$ is the linear part of $A_1$). Then $N_G=L_1^T(N_F)$. In particular, $N_F$ and $N_G$ are equivalent as point sets in $\PG(m-1,2)$. 
\end{theorem}
\begin{proof}
    The value of the Walsh transform $W_{G}(b,a)$ is uniquely determined by the number of solutions of $\langle b,G(x) \rangle +\langle a,x\rangle = 0$. After a transformation $x \mapsto A_2^{-1}(x),$ this number is exactly the same as the number of solutions of
    \[
        \langle b,A_1 \circ F(x) \rangle +\langle b,A_3 \circ A_2^{-1}(x) \rangle + \langle a,A_2^{-1}(x)\rangle = 0,
    \]
    which itself is equivalent to 
    \begin{align*}
       0&=\langle L_1^T(b),F(x) \rangle +\langle b,A_3 \circ A_2^{-1}(x) \rangle + \langle b,c_1 \rangle +\langle a,A_2^{-1}(x)\rangle\\
       &=\langle L_1^T(b),F(x) \rangle +\langle L_3^T(b)+a,A_2^{-1}(x)\rangle +\langle b, c_1+c_3\rangle,
    \end{align*}
    where $A_3=L_3+c_3$ with linear $L_3$. 
    So, the number of solutions of $\langle b,G(x) \rangle +\langle a,x\rangle = 0$ is the same as the number of solutions of $ \langle L_1^T(b),F(x) \rangle +\langle L_3^T(b)+a,A_2^{-1}(x)\rangle = k,$ where $k=\langle b, c_1+c_3\rangle \in \F_2$. In particular, if ${G}_b$ is bent, then so is $F_{L_1^T(b)}$, concluding the proof.
\end{proof} 

Similarly, we call two vector space partitions $\mathcal{V}_1=\{V_1,\dots,V_k\}$, $\mathcal{V}_2=\{W_1,\dots,W_l\}$ of $\F_2^n$ \emph{equivalent} if there is an $L \in \GL(n,2)$ such that for each $V_i$ there is a unique $W_j$ such that $V_i = L(W_j)$. 

\begin{theorem}
     Let $F,G \colon \F_2^n \rightarrow \F_2^n$ be two quadratic APN functions that are EA-equivalent via $G = A_1 \circ F \circ A_2 + A_3$. Then the two corresponding vector space partitions are equivalent.
\end{theorem}
\begin{proof}
    Let $L_1,L_2,L_3$ be the linear parts of $A_1,A_2,A_3$ respectively. Then
    \begin{align*}
        \image(D_{G,a})&= \{L_1 \circ F \circ A_2(x+a)+L_1 \circ F \circ A_2(x) + L_3(a) \colon x \in \F_2^n\} \\
        &=\{L_1 \circ F (x+L_2(a))+L_1 \circ F (x) + L_3(a) \colon x \in \F_2^n\} \\
        &=L_3(a)+L_1(\image(D_{F,L_2(a)})).
    \end{align*}
    So $\langle b, x\rangle$ is constant for all $x \in \image(D_{G,a})$ if and only if $\langle L_1^T(b), x\rangle$ is constant for all $x \in \image(D_{F,L_2(a)})$. In other words, $$\{a \in \F_{2}^n\colon \image(D_{G,a})\in \{H_b,\overline{H}_b\}\} = \{L_2^{-1}(a) \in \F_{2}^n\colon \image(D_{F,a})\in \{H_{L_1^T(b)},\overline{H}_{L_1^T(b)}\}\},$$
    proving the claim.
\end{proof}

\begin{remark}
    The APN property is also preserved under the more general notion of \emph{CCZ-equivalence}, however if $F$ is a quadratic function CCZ-equivalent to $G$, then $G$ is not necessarily quadratic, so our results cannot be transferred in this case.
\end{remark}
\subsection{New necessary conditions on when an APN function is CCZ-equivalent to a permutation}
The big APN question asks if APN permutations exist on $\F_2^{2m}$ for $m>3$. In~\cite{golouglu2021ccz}, the following necessary criterion was found to check when a function is CCZ-equivalent to a permutation.

\begin{theorem}[{\cite[Condition 2]{golouglu2021ccz}}]\label{thm:gp}
    Let $F\colon \F_2^n\rightarrow \F_2^n$ be CCZ-equivalent to a permutation. Then $N_F \cup \{0\}$ contains two subspaces $V,W \leq \F_2^n$, such that $V \oplus W = \F_2^n$.
\end{theorem}
The extra structure imposed by Theorem~\ref{thm:gp} can be used to improve the bound on $|N_F|$ given in the previous sections.

    Recall that we call a blocking set $B\subseteq \PG(m,q)$ with respect to $k$-spaces a $t$-fold blocking set if $B$ intersects every $k$-space in at least $t$ elements.
\begin{proposition}
    Let $n$ be even and $F \colon \F_2^n \rightarrow \F_2^n$ be a quadratic function. Assume $N_F \cup \{0\}$ contains two subspaces $V,W \leq \F_2^n$ such that $V \oplus W = \F_2^n$. Then $|N_F| \geq 3\cdot (2^{n/2}-1)$. Furthermore, $N_F$ is a $3$-fold blocking set of $\PG(n-1,2)$ with respect to $(n/2)$-spaces.
\end{proposition}
\begin{proof}
    Assume without loss of generality $n>\dim(V)\geq n/2$, and set $\dim(V)=k$. 
    Consider subspaces of $\F_2^n$ of dimension $k+1$ that contain $V$. There are exactly $2^{n-k}-1$ such vector spaces, and they are of the form $\langle V,w\rangle$ with $w \in W\setminus \{0\}$. Let $U=\langle V,w\rangle$ be such a subspace. Then $|(U\cap (V \cup W)) \setminus\{0\}|=2^{k}$. By Proposition~\ref{prop:odd}, $|N_F \cap U|$ is odd, so $N_F \cap U$ contains necessarily a vector outside of $V \cup W$. Since there are $2^{n-k}-1$ such $U$, $N_F$ contains $V\setminus\{0\}$, $W\setminus\{0\}$ and at least  $2^{n-k}-1$ more vectors outside of $V$ and $W$, i.e., $|N_F| \geq 2^{k}+2^{n-k}+2^{n-k}-3 \geq 3(2^{n/2}-1)$, proving the first statement.

    Let now $U\leq \F_2^n$ be an arbitrary subspace with $\dim(U)=n/2+1$. Then either $\dim(U \cap V)\geq 2$ (if $k\geq n/2+1$), or $U$ intersects both $V$ and $W$ nontrivially (if $k=n/2$). In both cases, $|U\cap N_F| \geq 2$, and by Proposition~\ref{prop:odd}, we have $|U\cap N_F| \geq 3$, so $N_F$ is a $3$-fold blocking set with respect to $(n/2)$-spaces.
\end{proof}
This immediately yields:
\begin{corollary} \label{cor:CCZ}
    Let $n$ be even and $F \colon \F_2^n \rightarrow \F_2^n$ be a quadratic function that is CCZ-equivalent to a permutation. Then  $|N_F| \geq 3\cdot (2^{n/2}-1)$ and $N_F$ is a $3$-fold blocking set of $\PG(n-1,2)$ with respect to $n/2$-spaces.
\end{corollary}

\section{Open Problems}
This paper establishes new connections between quadratic APN functions and certain vector space partitions, as well as blocking sets. This allowed us, among other things, to derive a variety of new conditions and bounds on the Walsh spectra of quadratic APN functions. These are the first known general conditions on the Walsh spectra outside of the well-known fourth-moment bound. 

We conclude this paper with several open problems that naturally arise from our results.

In Theorem~\ref{thm:dim8final}, we listed the amplitude distributions of APN functions $F\colon \F_2^8\rightarrow \F_2^8$ that do not violate the conditions we derived. However, we have not been able to find examples for many of the amplitude distributions listed in the theorem.
\begin{openproblem}
    Do there exist APN functions $F\colon \F_2^8 \rightarrow \F_{2}^8$ with the amplitude distributions $[0^{238-4i},2^{5i},4^{17-i}]$ for $1 \leq i \leq 12$?
\end{openproblem}
We have examples for the other amplitude distributions given in Theorem~\ref{thm:dim8final}. Of particular interest would be an amplitude distribution of type $[0^{234},2^{5},4^{16}]$. Such a function would have $21$ non-trivial non-bent components, which is exactly the lower bound given in Theorem~\ref{thm:APN bent upper bound}. In particular, the blocking set $N_F$ is in this case exactly a smallest possible non-trivial blocking set (see Theorem~\ref{thm:govaerts}). These blocking sets necessarily have a specific structure, namely they are a specific cone in $\PG(n-1,2)$, see~\cite[Theorem 1.4.(3)]{govaerts2006classification}.

More generally, given the lack of examples for some amplitude distributions in the dimension 8 case, we pose the following open problems.
\begin{openproblem}
    Improve the bounds on $|N_F|$ we found in this paper for quadratic APN functions, or find examples that satisfy the bounds with equality.
\end{openproblem}

\begin{openproblem}
    Find further examples of vector space partition types that are not associated to any quadratic APN function.
\end{openproblem}

In Theorem~\ref{thm:generalbentbound}, we proved that if $F \colon \F_2^n\rightarrow \F_2^m$ is quadratic and not MNBC, then $|N_F| \geq 2^{m-n/2}+2^{m-n/2-2}+2^{m-n/2-3}-1$. We currently do not know of an example that satisfies this bound with equality. Again, in this case, $N_F$ necessarily has the structure of a specific cone in $\PG(m-1,2)$ (see again~\cite[Theorem 1.4.(3)]{govaerts2006classification}). 

\begin{openproblem}
    Let $n$ be even, $n \leq 2m-6$. Does there exist a quadratic function $F \colon \F_2^n\rightarrow \F_2^m$ with exactly $2^{m-n/2}+2^{m-n/2-2}+2^{m-n/2-3}-1$ non-trivial non-bent components?
\end{openproblem}

Of course, a big open problem is to find an infinite family of quadratic APN functions that have a non-classical Walsh spectrum (or equivalently, are associated to a vector space partition that is not a $2$-spread). This is a well known important research question in the study of APN functions. We hope that the results on possible amplitudes/Walsh spectra in this paper can give new ideas and impulses for this question as well.

\subsection*{Acknowledgments}
The majority of this work was done when S.B. and N.G. visited the University of South Florida in summer 2025 as part of a Research Experience for Undergraduates program, funded by NSF grant number 2244488. The authors gracefully acknowledge the financial support provided by this program.


\bibliographystyle{IEEEtran}
 \bibliography{references} 
\end{document}